\documentclass{article}

\usepackage{amsmath}
\usepackage{amssymb}
\usepackage{amsthm}

\newtheorem{thm}{Theorem}
\newtheorem{lem}[thm]{Lemma}
\newtheorem{cor}[thm]{Corollary}
\newtheorem{prop}[thm]{Proposition}

\theoremstyle{remark}
\newtheorem{remark}[thm]{Remark}

\theoremstyle{definition}
\newtheorem{definition}[thm]{Definition}

\newenvironment{BaileyL}{\vskip 2mm\noindent \textbf{Bailey's Lemma} \em}{}{}
\newenvironment{RogRamA}{\vskip 2mm\noindent \textbf{The Rogers-Ramanujan Identities}
 \em}{}{}
\newenvironment{GordonThm}{\vskip 2mm \noindent\textbf{Gordon's Partition Theorem}\em}{}{}
\newenvironment{AndrewsAnalytic}{\vskip 2mm \noindent\textbf{
   Andrews' Analytic Counterpart to Gordon's
  Theorem}\em}{}{}

\newcommand{\hgs}[6]{ {}_{#1}\phi_{#2} \left[ \genfrac{}{}{0pt}{}{#3}{#4} ; {#5},{#6} \right]}
\newcommand{\binomial}[2]{ \genfrac{(}{)}{0pt}{}{#1}{#2} }

\allowdisplaybreaks[1]
\numberwithin{thm}{section}

\begin{document}




\author{Andrew V. Sills}
\title{On Identities of the Rogers--Ramanujan Type}
\date{Received May 2003; Revised June 1, 2004}

\maketitle


\begin{abstract}A generalized Bailey pair, which contains several 
special cases considered
by Bailey (\emph{Proc. London Math. Soc. (2)}, 50 (1949), 421--435), is
derived and used to find a number of new Rogers-Ramanujan type identities.
Consideration of associated $q$-difference equations points to a 
connection with a mild extension of Gordon's combinatorial generalization
of the Rogers-Ramanujan identities (\emph{Amer. J. Math.}, 83 (1961),     
393--399).  This, in turn, allows the
formulation of natural combinatorial interpretations of many of
the identities in Slater's list (\emph{Proc. London Math. Soc. (2)} 54 (1952),
147--167), as well as the new identities presented here.  A list of
26 new double sum--product Rogers-Ramanujan type identities are 
included as an appendix. 
\end{abstract}


\section{Introduction}\label{intro}
\subsection{Overview} 
We begin by recalling the famous Rogers-Ramanujan identities:
\begin{RogRamA} 
 \begin{equation}\label{RRa1}
   \sum_{n=0}^\infty \frac{q^{n^2}}{(q;q)_n} = 
   \frac{(q^2, q^3, q^5; q^5)_\infty}{(q;q)_\infty},
 \end{equation} and
\begin{equation}\label{RRa2}
   \sum_{n=0}^\infty \frac{q^{n^2+n}}{(q;q)_n} = 
   \frac{(q, q^4, q^5; q^5)_\infty}{(q;q)_\infty},
 \end{equation}
where \[ (a;q)_m = \prod_{j=0}^{m-1} (1-aq^j), \]
      \[ (a;q)_\infty = \prod_{j=0}^\infty (1-aq^j), \] and
      \[ (a_1, a_2, \dots, a_r; q)_s = (a_1;q)_s (a_2;q)_s \dots (a_r;q)_s, \]
and throughout this paper we assume $|q|<1$ to ensure convergence.
\end{RogRamA} 

The Rogers-Ramanujan identities are due to L.~J.~Rogers~\cite{Rog1894},
and were rediscovered independently by S. Ramanujan~\cite{MacMahon} and
I. Schur~\cite{Schur}.  In the 1940's, W.~N.~Bailey undertook a careful
study of Rogers' work, and greatly simplified it in a pair of 
papers~(\cite{Bailey1} and~\cite{Bailey2}).  In these papers, Bailey was
able to prove what he termed ``$a$-generalizations'' 
(i.e. formulae with a second variable $a$ in addition to $q$), of the Rogers-Ramanujan 
identities and a number of additional identities 
of similar type (some of which were due to Rogers and others of
which were new at the time).  Hereafter, $a$-generalizations of Rogers-Ramanujan
type identities will be referred to simply as ``$a$-RRT identities.''

 By considering a certain ``parametrized Bailey pair,'' we will be 
naturally led to a variety of $a$-RRT identities, some of which were
found by Bailey, and others of which appear to be new.  Some examples of
new $a$-RRT identities include 
\begin{gather} 
  \sum_{n\geqq 0} \frac{a^n q^{n(n+1)/2} (-1;q)_n}{(aq;q^2)_n (q;q)_n}\nonumber\\
 = \frac{(-aq;q)_\infty}{(aq;q)_\infty}
     \sum_{r\geqq 0} \frac{(-1)^r a^{3r} q^{5r^2} (-1;q)_{2r} (1-aq^{4r}) (a;q^2)_r }
          { (1-a) (q^2;q^2)_r (-aq;q)_{2r} }\label{a-mod10SS}
\end{gather}
  and
\begin{gather}
 \sum_{n\geqq 0} \sum_{r\geqq 0} \frac{a^{n+r} q^{n^2 + 2r^2}}
 {(aq;q^2)_n (q^2;q^2)_r (q;q)_{n-r} }\nonumber\\
 = \frac{1}{(aq;q)_\infty }
 \sum_{r\geqq 0}\frac{(-1)^r a^{4r} q^{9r^2-r} (1-aq^{4r}) (a;q^2)_r}
 {(1-a)(q^2;q^2)_r} \label{a-mod18i}
 \end{gather}
 
From the $a$-RRT identities, such as \eqref{a-mod10SS} and \eqref{a-mod18i},
we may easily deduce elegant Rogers-Ramanujan type identities (in $q$ only); 
in these instances we obtain:
\begin{equation} \label{mod10SS}
 \sum_{n\geqq 0} \frac{ q^{n(n+1)/2} (-1;q)_n}{ (q;q^2)_n (q;q)_n}
=\frac{(q^5,q^5,q^{10};q^{10})_\infty (-q;q)_\infty}{(q;q)_\infty},
\end{equation}
which, surprisingly is not included in Slater's list~\cite{Slater2}, and
\begin{equation} \label{mod18}
   \sum_{n\geqq 0}\sum_{r\geqq 0} 
    \frac{q^{n^2 + 2r^2} }{(q;q^2)_n (q^2;q^2)_r (q;q)_{n-2r}}
  = \frac{(q^{8},q^{10},q^{18};q^{18})_\infty}{(q;q)_\infty}.
\end{equation}

\begin{remark} The referee pointed out that (\ref{a-mod10SS}) follows from the
$e, d\to\infty$, $c=-1$ case of \cite[p. 68, (3.5.7)]{GR} and that Bailey
actually had a generalization of (\ref{a-mod10SS}), namely 
\cite[p. 6 (6.3)]{Bailey2}, which makes it all the more remarkable that
(\ref{mod10SS}) did not appear in Slater's list.  As we shall see later, 
(\ref{a-mod10SS}) and (\ref{mod10SS}) follow from the $(d,k)=(2,3)$ case
of the parametrized Bailey pair, and (\ref{a-mod18i}) and (\ref{mod18}) 
follow from the $(d,k)=(2,4)$ case.
\end{remark}

  Once we have an $a$-RRT identity in hand, we then study the $q$-difference
equations related to the associated set of of identities.  Observing the
patterns which emerge in the $q$-difference equations associated with
various sets of identities, one is led to consider the following mild
extension of Basil Gordon's partition theorem:

\begin{thm}\label{GenGordon}
Let $A_{d,k,i}(n)$ denote the number of partitions of $n$ into parts 
$\not\equiv 0, \pm di\pmod{2dk+d}$.  Let $B_{d,k,i}(n)$ denote the
number of partitions of $n$ wherein
  \begin{itemize}
     \item The integer $d$ appears as a part at most $i-1$ times,
     \item the total number of appearances of $dj$ and $dj+d$ (i.e. any two consecutive
       multiples of $d$) together is at most $k-1$, and
     \item nonmultiples of $d$ may appear as parts without restriction.
  \end{itemize}
Then for $1\leqq i\leqq k$, $A_{d,k,i}(n) = B_{d,k,i}(n)$.
\end{thm}

\begin{remark}
The case $d=1$ is Gordon's partition theorem~\cite{Gordon}.\end{remark}

  As we shall see, special cases of Theorem~\ref{GenGordon} provide
new combinatorial interpretations for various identities in Slater's
list~\cite{Slater2}, as well as for the new analytic identities presented
here.

 For example, consider the Rogers mod 14 identities, which appear in
Slater~\cite{Slater2} as identities (59), (60), and (61) 
(see \eqref{RogMod14-3}--\eqref{RogMod14-1}).  We shall
see that these may be interpreted combinatorially as the $d=2$, $k=3$
case of Theorem~\ref{GenGordon}:

\begin{cor} \label{combmod14} For $i=1,2,3$,
the number of partitions of $n$ into parts wherein
   \begin{itemize}
     \item $2$ appears as a part at most $i-1$ times,
     \item the total number of appearances of any two consecutive even
numbers is at most $2$, and
     \item odd numbers may appear as parts without restriction,   
    \end{itemize}
equals the number of partitions of $n$ into parts not congruent to
$0,\pm 2i\pmod{14}$.
\end{cor}

Similarly, the combinatorial interpretation of \eqref{mod18} is
\begin{cor}  The number of partitions of $n$ into parts wherein
   \begin{itemize}
     \item $2$ appears as a part at most $3$ times,
     \item the total number of appearances of any two consecutive 
even numbers is at most $3$, and
     \item odd numbers may appear as parts without restriction,   
    \end{itemize}
equals the number of partitions of $n$ into parts not congruent to
$0,\pm {8}\pmod{18}$.
\end{cor}

\subsection{Background}
The part of Bailey's results necessary for this current discussion may
be briefly summarized as follows:

\begin{definition}
A pair of sequences $\left(\alpha_n (a,q),\beta_n(a,q)\right)$ is called a \emph{Bailey pair} if
for $n\geqq 0$, 
   \begin{equation} \label{BPdef}
      \beta_n (a,q) = \sum_{r=0}^n \frac{\alpha_r(a,q) }{(q;q)_{n-r} (aq;q)_{n+r}}.
   \end{equation}
\end{definition}
  
In~\cite{Bailey1} and~\cite{Bailey2}, Bailey proved the fundamental
result now known as ``Bailey's Lemma'' (see also~\cite[Chapter 3]{GEA:qs}):
\begin{BaileyL}
If $(\alpha_r (a,q), \beta_j (a,q))$ form a Bailey pair, then
\begin{gather} 
  \frac{1}{ (\frac{aq}{\rho_1};q)_n ( \frac{aq}{\rho_2};q)_n} 
  \sum_{j\geqq 0} \frac{ (\rho_1;q)_j (\rho_2;q)_j 
     (\frac{aq}{\rho_1 \rho_2} ;q)_{n-j}}
     {(q;q)_{n-j}} \left( \frac{aq}{\rho_1 \rho_2} \right)^j \beta_j(a,q) \nonumber\\
 = \sum_{r=0}^n \frac{ (\rho_1;q)_r (\rho_2;q)_r}
    { (\frac{aq}{\rho_1};q)_r (\frac{aq}{\rho_2};q)_r (q;q)_{n-r} (aq;q)_{n+r}}
    \left( \frac{aq}{\rho_1 \rho_2} \right)^r \alpha_r(a,q). \label{BL}
\end{gather}
\end{BaileyL}

An immediate consequence of Bailey's Lemma is the following important corollary:

\begin{cor}
If $(\alpha_m (a,q), \beta_j (a,q))$ form a Bailey pair, then
\begin{equation} \label{WBL}
  \sum_{j\geqq 0} a^j q^{j^2} \beta_j (a,q) 
   = \frac{1}{(aq;q)_\infty} \sum_{m=0}^\infty a^m q^{m^2} \alpha_m (a,q),
\end{equation}
\begin{gather} 
   \sum_{j\geqq 0} a^j q^{j^2} (-q;q^2)_j \beta_j (a,q^2) 
  =\frac{(-aq;q^2)_\infty}{(aq^2;q^2)_\infty} 
   \sum_{m=0}^\infty \frac{ a^m q^{m^2} (-q;q^2)_m} {(-aq;q^2)_m}
    \alpha_m(a,q^2), \label{ATNSBL}
  \end{gather} and 
\begin{gather} 
   \sum_{j\geqq 0} a^j q^{j(j+1)/2} (-1;q)_j \beta_j (a,q)
  =\frac{(-aq;q)_\infty}{(aq;q)_\infty} 
   \sum_{m=0}^\infty \frac{ a^m q^{m(m+1)/2} (-1;q)_m} {(-aq;q)_{m}}
    \alpha_m (a,q). \label{SSBL}
  \end{gather} 
\end{cor}
\begin{proof}  First, let $n,\rho_1\to\infty$ in \eqref{BL}.  Then, to obtain
\eqref{WBL}, let $\rho_2\to\infty$; to obtain \eqref{SSBL}, set $\rho_2=-1$;
and finally to obtain \eqref{ATNSBL}, set $\rho_2=-\sqrt{q}$, and then replace
$q$ by $q^2$ throughout.\openbox
\end{proof}

Thus the substitution of any Bailey pair 
$(\alpha_n(a,q),\beta_n(a,q))$ into 
\eqref{WBL}, \eqref{ATNSBL}, or \eqref{SSBL} yields an $a$-RRT identity.
Bailey did exactly this in~\cite{Bailey1} and~\cite{Bailey2}.  
Setting $a=1$ or $a=q$, one obtains traditional Rogers-Ramanujan type
identities in the variable $q$ only.  Bailey's
student L.J. Slater \cite{Slater2} obtained a list of 
130 Rogers-Ramanujan type
identities (in $q$ only) in precisely this way.
   In \S\ref{GenBP}, we study a general Bailey pair for which several
special cases were considered by Bailey himself in~\cite{Bailey2}.  Next,
in \S\ref{qDiffEqns}, we derive $q$-difference equations for various sets
of $a$-RRT identities, and consider their partition
theoretic implications in \S\ref{PtnThms}.  The narrative is concluded with
some observations and open questions in \S\ref{concl}.  Finally,
an appendix containing 26 new double sum--product Rogers-Ramanujan
type identities is included.

\section{A Parametrized Bailey Pair}\label{GenBP}
In~\cite{Bailey1} and ~\cite{Bailey2}, Bailey considered several Bailey pairs
which are special cases of a more general Bailey pair involving additional
parameters $d$ and $k$:
\begin{thm}\label{ParamBP}
Let $\lambda = -\frac{3}{2}d^2 + dk + \frac 12 d$, $h = |\frac{2\lambda}{d}|$,
and  $t=d+h+2$.   Let
\begin{equation*} 
    \alpha_{d,k,m}(a,q) := 
        \left\{  \begin{array}{ll}
     \parbox{4cm}{\[ \frac{(-1)^r a^{(k-d)r} q^{(dk - d^2 + \frac d2)r^2 - \frac d2 r} 
       (aq^{2d};q^{2d})_r (a;q^d)_r}{(a;q^{2d})_r (q^d;q^d)_r }, \]}\\
            &\mbox{if $m= dr$, and} \\
      0,                                      &\mbox{otherwise,}
              \end{array} \right.
    \end{equation*}
and \begin{equation*}
    \hskip -2cm \beta_{d,k,m}(a,q) := 
      \left\{  \begin{array}{lr}
        \parbox{4cm}{\[
         \lim_{\tau\to 0}
         \frac{ {}_{t+1}W_{t}(a; \nu_1,\dots, \nu_h, \mu_1,\dots,\mu_d; q^d; \tau^h a^{k-d} q^{nd})}
              {(q,aq;q)_n}\]} & \\
          & \mbox{if $\lambda\geqq 0$,} \\
         \parbox{4cm}{\[
         \lim_{\tau\to 0}
         \frac{{}_{t+1}W_{t}(a; \delta_1,\dots, \delta_h,\mu_1,\dots,\mu_d;q^d; 
               \frac{a^{k-d} q^{nd}}{\tau^h})}{(q,aq;q)_n} \]} & \\
          & \mbox{if $\lambda< 0$,} 
              \end{array} \right.
    \end{equation*}
where $\nu_j = \frac{q^{\lambda/h}}{\tau}$, $\mu_j = q^{d-j-n}$,
$\delta_j = \tau a q^{d-\lambda/h}$, 
\begin{gather*} 
  {}_{s+1}W_{s}(a_1;a_4,a_5,\dots,a_{s+1};q,z) 
  = \hgs{s+1}{s}{a_1,qa_1^{\frac 12},-qa_1^{\frac 12},a_4,\dots,a_{s+1}}
       {a_1^{\frac 12},-a_1^{\frac 12},\frac{qa_1}{a_4},\dots,\frac{qa_1}{a_{s+1}}}{q}{z},
\end{gather*} and
\begin{equation*}
  \hgs{s+1}{s}{a_1,a_2,\dots,a_{s+1}}{b_1,b_2,\dots,b_s}{q}{z} =
  \sum_{r=0}^\infty \frac{ (a_1,a_2,\dots,a_{s+1};q)_r}{(q,b_1,b_2,\dots,b_s;q)_r} z^r.   
\end{equation*}
Then $(\alpha_{d,k,m}(a,q), \beta_{d,k,n}(a,q))$ form a Bailey pair.
\end{thm} 
 
\begin{remark} The notation above is quite dense, and so a few words of 
clarification are perhaps in order.  
$\lambda$ represents the co\"efficient of $r^2$ in the exponent of $q$ which arises
when $\alpha_{d,k,m}(a,q)$ is inserted into the RHS of~\eqref{BPdef}. $h$ is the number
of rising $q$-factorials necessary to write $q^{\lambda r^2}$ as a limit as $\tau\to 0$
of a power of $\tau$ times the rising $q$ factorials in base $q^{d}$.  For example, to
write $q^{4r^2}$ using base $q^2$, we find $h=4$ since 
 \[ q^{4r^2} = \lim_{\tau\to 0} \tau^{4r} (q/\tau;q^2)_r^4. \]  
$t$ is the total number of denominator entries in the resulting very-well poised
basic hypergeometric series.
\end{remark}

\begin{proof}[Proof of Theorem~\ref{ParamBP}]
\begin{eqnarray*}
  &&\beta_{d,k,n}(a,q)\\
 &=& \sum_{m=0}^n \frac{1}{(q;q)_{n-r} (aq;q)_{n+r}} 
\alpha_{d,k,m}(a,q)\\
    &=& \frac{1}{(q;q)_{n} (aq;q)_{n}} 
      \sum_{m=0}^n \frac{(-1)^m q^{nm +\frac m2 -\frac{m^2}{2}} (q^{-n};q)_m}{(aq^{n+1};q)_m} 
      \alpha_{d,k,m}(a,q) \\
    &=& \frac{1}{(q;q)_{n} (aq;q)_{n}} 
      \sum_{r=0}^{\lfloor n/d \rfloor} 
       \frac{(-1)^{dr} q^{ndr +\frac d2 r -\frac{d^2 }{2} r^2} (q^{-n};q)_{dr}}{(aq^{n+1};q)_{dr}} 
      \alpha_{d,k,dr}(a,q) \\
    &=& \sum_{r=0}^{\lfloor n/d \rfloor} 
       \frac{(-1)^{(d+1)r} a^{(k-d)r} q^{\lambda r^2 + ndr} 
          (a;q^d)_r (a q^{2d};q^{2d})_r (q^{-n};q)_{dr} }
       {(q;q)_{n} (aq;q)_{n} (q^d;q^d)_r (a; q^{2d})_r (aq^{n+1};q)_{dr}}
\end{eqnarray*}
If $\lambda\geqq 0$, this last expression
\begin{gather*}
  =\frac{1}{(q;q)_n (aq;q)_n}
    \lim_{\tau\to 0} 
       \sum_{r\geqq 0} \left\{ \frac{(a,q^d\sqrt{a},-q^d\sqrt{a};q^d)_r 
          (q^{\lambda/h}/\tau;q^d)_r^h}
          {(q^d,\sqrt{a},-\sqrt{a};q^d)_r 
           (\tau a  q^{d-\lambda/h};q^d)_r^h} \right. \\
     \qquad\qquad \left. \times\frac {(q^{d-1-n},q^{d-2-n},\dots, q^{-n};q^d)_r}
           {(aq^{n+1},aq^{n+2},\dots,aq^{n+d};q^d)_r} \tau^{hr} a^{(k-d)r} q^{ndr} \right\},
\end{gather*}
while if $\lambda<0$, we instead place $q^{-\lambda r^2}$ in the denominator:
\begin{gather*}
=\frac{1}{(q;q)_n (aq;q)_n} 
     \lim_{\tau\to 0} 
       \sum_{r\geqq 0} \left\{ \frac{ (a,q^d\sqrt{a},-q^d\sqrt{a};q^d)_r 
          (\tau a q^{d-\lambda/h};q^d)_r^h}
          {(q^d,\sqrt{a},-\sqrt{a};q^d)_r  
           (q^{\lambda/h}/\tau;q^d)_r^h} \right. \\
     \qquad\qquad \left. \times\frac {(q^{d-1-n},q^{d-2-n},\dots, q^{-n};q^d)_r}
           {(aq^{n+1},aq^{n+2},\dots,aq^{n+d};q^d)_r} \tau^{hr} a^{(k-d)r} q^{ndr} \right\}.
\end{gather*}\openbox
\end{proof}

The goal is
to find Bailey pairs which will give rise to attractive identities.
Bailey himself considered the special cases 
$\alpha_{d,k,m}(a,q)$ for $(d,k) = (1,2), (2,2), (2,3)$, and $(3,4)$~\cite[p. 5--6,
eqns. (i), (iv) with $f=0$, (iv) with $f\to\infty$, and (v) 
respectively]{Bailey2}.
Each of these four $(d,k)$ sets is particularly nice, as the
resulting expression for $\alpha_{d,k,r}(a,q)$, when substituted into~\eqref{BPdef},
is a finite product times a ${}_6W_5$ on base $q^d$, 
which is summable by Jackson's theorem~\cite[p. 238, eqn. (II.20)]{GR}.
Thus, $\beta_{d,k,n}(a,q)$ reduces to a finite product, and upon substituting it into
\eqref{WBL}, the left hand side of the resulting $a$-RRT
identity will be a single-fold sum.  
  
  In this way, upon letting $a\to 1$,
we may derive the first Rogers-Ramanujan 
identity~\eqref{RRa1} from $(d,k) = (1,2)$,
a Rogers' mod $10$ identity~\eqref{RogMod10-1}
from $(d,k) = (2,2)$,
a Rogers mod $14$ identity~\eqref{RogMod14-1}
from $(d,k) = (2,3)$,
and a Bailey-Dyson mod $27$ identity~\eqref{BaileyMod27-1}
from $(d,k) = (3,4)$. It was not mentioned by Bailey, but Euler's 
pentagonal number theorem~\cite[p. 11, Cor. 1.7]{GEA:top}
arises from the case $(d,k) = (1,1)$.  Similarly, by substituting
the Bailey pairs into \eqref{ATNSBL} and \eqref{SSBL}, and then
letting $a\to 1$, other identities from Slater's list may be derived.
One case that both Bailey and Slater seem to have missed is
the substitution of $(d,k)=(2,3)$ into \eqref{SSBL}, which immediately
yields \eqref{a-mod10SS} and then \eqref{mod10SS} when $a=1$.

  Note that, in fact, $d=1$ corresponds to the ``unit Bailey 
chain''~\cite{GEA:multiRR}.  Substituting the Bailey pairs corresponding
to the $d=1$ cases into \eqref{WBL} yields cases of Andrews' analytic 
generalization of the Rogers-Ramanujan identities for odd 
moduli~\cite{GEA:oddmoduli}; see \eqref{AGI}.

  Thus to search for new identities, we need to consider $d>1$.  Also, in order
to find $\beta_{d,k,n}$'s with relatively simple forms, 
$d+h$ should be kept
as small as possible since $\beta_n$ is a finite product times a 
${}_{d+h+3}W_{d+h+2}$, and the higher one looks in the hypergeometric
hierarchy, the more complicated things become.  
It appears that Bailey considered
all cases where $d+h=3$, and thus all of the summable ${}_6 W_5$'s.
The next best situation is where $d+h=5$, which corresponds to a
${}_8 W_7$ that can be transformed by Watson's 
$q$-analog of Whipple's Theorem~\cite[p. 242, eqn. (III.17)]{GR}:

  Consider the case $(d,k)=(2,4)$:
 \begin{eqnarray}
  &&\beta_{2,4,n}(a,q) \nonumber \\
  &=& \frac{1}{(q;q)_n (aq;q)_n}
       \sum_{r=0}^{\lfloor n/2 \rfloor} 
       \frac{(-1)^{r} a^{2r} q^{3r^2 + 2nr} 
          (a;q^2)_r (a q^{4};q^{4})_r (q^{-n};q)_{2r} }
       {(q^2;q^2)_r (a; q^{4})_r (aq^{n+1};q)_{2r}}\nonumber\\
   &=& \frac{1}{(q;q)_n (aq;q)_n} \nonumber
   \\ &&\times \lim_{\tau\to 0} 
    \hgs{8}{7}{a,q^2 \sqrt{a},-q^2 \sqrt{a}, \frac{q}{\tau}, 
      \frac{q}{\tau}, \frac{q}{\tau}, q^{1-n}, q^{-n}}
       {\sqrt{a},-\sqrt{a}, \tau aq, \tau aq, \tau a q, aq^{n+1},aq^{n+2}}{q^2}
        {\tau^{3} a^{2} q^{2n} }\nonumber\\
   &=& \frac{1}{(q;q)_n (aq;q)_n} \lim_{\tau\to 0} 
     \frac{ (aq^2,\tau aq^{n+1},\tau aq^n, aq^{2n+1}; q^2)_\infty}
     {(aq\tau, a q^{n+2}, a q^{n+1}, \tau a q^{2n};q^2)_\infty}\nonumber\\ &&\times
     \hgs{4}{3}{\tau^2 a, \frac{q}{\tau},q^{-n},q^{1-n}}
         {\tau a q, \tau a q, \frac{q^{2-2n}}{a\tau}}{q^2}{q^2}
  \qquad\mbox{\quad (by~\cite[p. 242, eqn. (III.17)]{GR})}\nonumber\\
   &=& \frac{(aq^2,aq^{2n+1};q^2)_\infty}{(q,aq;q)_n (aq^{n+1};q)_\infty}
     \sum_{r\geqq 0} \frac{(q^{-n};q)_{2r}}{(q^2;q^2)_r} a^r q^{2nr+r}\nonumber\\
   &=& \frac{1}{(aq;q^2)_n}
     \sum_{r\geqq 0} \frac{a^r q^{2r^2}}{(q^2;q^2)_r (q;q)_{n-2r}}. \label{beta24}
     \end{eqnarray}  
  
Analogous calculations allow us to find
\begin{eqnarray}
\beta_{2,1,n}(a,q) &=& \frac{q^{\binomial{n}{2}} }{(aq;q^2)_n} 
  \sum_{r\geqq 0} \frac{(-1)^r a^{-r} q^{r^2 - 2nr}}{(q^2;q^2)_r (q;q)_{n-2r}} 
  \label{beta21}\\
\beta_{3,3,n}(a,q) &=& \frac{1}{(a;q)_{2n} }
   \sum_{r\geqq 0}\frac{(-1)^r q^{\frac 32 r^2 - \frac 32 r} (a;q^3)_{n-r}}
     {(q^3;q^3)_r (q;q)_{n-3r}} \label{beta33}\\
\beta_{3,5,n}(a,q) &=& \frac{1}{(a;q)_{2n} }
   \sum_{r\geqq 0}\frac{a^r q^{3r^2} (a;q^3)_{n-r}}
     {(q^3;q^3)_r (q;q)_{n-3r}}. \label{beta35}
\end{eqnarray}

Once $d>3$, even if $d+h=5$, Watson's $q$-Whipple transformation~\cite[p. 242, eqn. (III.17)]{GR} 
is not applicable as the
resulting ${}_4\phi_3$ does not terminate.  In this case, we must use the more
general transformation~\cite[p. 246, eqn. (III.36)]{GR}.  
Let us now consider such a situation:

\begin{eqnarray}
&& \beta_{4,6,n} (a,q) \nonumber\\
&=& \frac{1}{(q,aq;q)_n} \lim_{\tau\to 0} 
   {}_8W_7 \left[a;\frac{q^2}{\tau^2},q^{1-n},q^{3-n},q^{-n},q^{2-n};q^4;
   \tau^2 a^2 q^{4n} \right] 
   \nonumber \\
   &=& \lim_{\tau\to 0} \frac{(aq^2,a q^{n-1}\tau, a q^{2n+1}, -a q^{2n+1}, 
      \tau^{\frac 12} a q^{n+1}, -\tau^{\frac 12}a q^{n+1})_\infty}
  {(q,aq;q)_n (a\tau, a q^{n+1}, a q^{n+2}, -aq^{n+2}, \tau a q^{2n}, -\tau a q^{2n}; q^2)_\infty}
  \nonumber \\
  &&\times \hgs{8}{7}
    {-aq^n,ia^{\frac 12} q^{2+\frac n2} ,-i a^{\frac 12} q^{2+\frac n2}, q^{1-n},\frac{q}{\tau},
  -\frac{q}{\tau},a^{\frac 12} q^{1+n},-a^{\frac 12} q^{1+n} }
  {ia^{\frac 12} q^{\frac n2},-ia^{\frac 12} q^{\frac n2},-aq^{2n+1},-\tau aq^{n+1}, \tau a q^{n+1},
  -qa^{\frac 12},qa^{\frac 12}}{q^2}{a q^{n-1}} \nonumber \\
  & &  \qquad\qquad\qquad\qquad\mbox{(by~\cite[p. 70, eqn. (3.5.10)]{GR})}\nonumber\\
  &=& \lim_{\tau\to 0} \frac{(aq^2,\tau^2 a q^{n-1},a q^{2n+1}, -a q^{2n+1}, \tau a q^{n+1}
  -\tau a q^{n+1};q^2)_\infty}
  {(q,aq;q)_n (\tau^2 a,a q^{n+1},a q^{n+2},-a q^{n+2}, \tau a q^{2n}, -\tau a q^{2n};q^2)_\infty}
  \nonumber \\
  && \times \left\{ \frac{(-aq^{n+2},q^{-n},-q^n a^{\frac 12},q^n a^{\frac 12};q^2)_\infty}
        {(-qa^{\frac 12},qa^{\frac 12},-a q^{2n+1}, \frac 1q; q^2)_\infty}
    \right. \nonumber \\
     && \qquad \times
  \hgs{4}{3}{\tau^2 a q^n, a^{\frac 12}q^{n+1},-a^{\frac 12} q^{n+1},q^{1-n}}
       {-\tau a q^{n+1},-\tau a q^{n+1}}{q^2}{q^2} \nonumber\\
  && \qquad+\frac{(-aq^{n+2},\tau^2 a q^n, a^{\frac 12} q^{n+1}, -a^{\frac 12} q^{n+1}, -\tau a q^n,
    \tau a q^n;q^2)_\infty}
    {(-\tau a q^{n+1},\tau a q^{n+1},-qa^{\frac 12} ,qa^{\frac 12}, -a q^{2n+1},a q^{2n+1},
      \tau^2 a q^{n-1},q;q^2)_\infty}\nonumber\\
   &&  \qquad\left. \times\hgs{4}{3}{q^{-n},-q^n a^{\frac 12},q^n a^{\frac 12},\tau^2 a q^{n-1}}
     {-\tau a q^n, \tau a q^n, q}{q^2}{q^2} \right\}\nonumber\\
   && \qquad\qquad\qquad\qquad\mbox{(by \cite[p. 246, eqn. (III.36)]{GR})}
     \nonumber\\
  &=& \frac{1}{(q;q)_n (aq;q^2)_n (aq^2;q^4)_\infty (q;q^2)_\infty} \nonumber \\
  && \times\left\{ -(q^{-n};q^2)_\infty (aq^{2n};q^4)_\infty \sum_{r\geqq 0} 
     \frac{(aq^{2n+2};q^4)_r (q^{1-n};q^2)_r q^{2r+1}}{(q;q)_{2r+1}} \right. \nonumber\\
  && \left. \qquad +(q^{1-n};q^2)_\infty (a q^{2n+2};q^4)_\infty \sum_{r\geqq 0}
      \frac{(aq^{2n};q^4)_r (q^{-n};q^2)_r q^{2r}}{(q;q)_{2r}} \right\} \nonumber
\end{eqnarray}
Notice that the first term vanishes for $n$ even and the second for $n$ odd.
Thus we conclude 
\begin{equation} \label{beta46even} 
\beta_{4,6,2m}(a,q) = \sum_{r\geqq 0} \frac{(-1)^{m+r} q^{r^2-m^2 + r - 2mr} (a;q^4)_{m+r}}
 {(a;q)_{4m} (q;q)_{2r} (q^2;q^2)_{m-r}} 
\end{equation} and
\begin{equation} \label{beta46odd}
\beta_{4,6,2m+1}(a,q) = \sum_{r\geqq 0} \frac{(-1)^{m+r} q^{r^2-m^2 + r - 2m-2mr} (a;q^4)_{m+r+1}}
 {(a;q)_{4m+2} (q;q)_{2r+1} (q^2;q^2)_{m-r}}.
\end{equation}

\section{$q$-difference equations} \label{qDiffEqns} 
For each of the Bailey pairs derived in \S\ref{GenBP}, we are able to obtain one
$a$-RRT identity from each of \eqref{WBL}, \eqref{ATNSBL}, and \eqref{SSBL}.  
However, in general there are a 
set of $k$ identities associated with a given $(d,k)$.  We will use $q$-difference equations
to establish complete sets of $k$ identities for various $(d,k)$ considered in \S\ref{GenBP},
as well as those $(d,k)$ considered by Bailey~\cite{Bailey2}.

\subsection{Expressions for the right hand sides and their $q$-difference equations}
\begin{definition} For $k\geqq 1$, and $1\leqq i \leqq k$,
\begin{gather} 
  Q_{d,k,i}(a) := Q_{d,k,i}(a,q)\nonumber\\:=\frac{1}{(aq;q)_\infty} 
    \sum_{n\geqq 0} \frac{(-1)^n a^{kn} q^{(dk + \frac d2)n^2 +(k-i+\frac 12)dn} 
       (1-a^i q^{(2n+1)di})(aq^d;q^d)_n}
    {(q^d;q^d)_n}.\label{Qdef}
\end{gather}
\end{definition}

\begin{thm}\label{qdiffs}
The following $q$-difference equations are valid:
\begin{equation}
 Q_{d,k,1}(a) = \frac{1}{(aq;q)_{d-1}} Q_{d,k,k}(aq^d) \label{qdiff1} 
\end{equation} and for $2\leqq i \leqq k$,
\begin{equation} \label{qdiff2}
Q_{d,k,i}(a) = Q_{d,k,i-1}(a) + \frac{a^{i-1} q^{(i-1)d}}
{(aq;q)_{d-1}} Q_{d,k,k-i+1}(aq^d).
\end{equation}
\end{thm}

Before proving Theorem~\ref{qdiffs}, we need the following lemma:
\begin{lem} \label{Qdkk}
\begin{equation*} 
  Q_{d,k,k}(a) = \frac{1}{(aq;q)_\infty} 
    \sum_{n\geqq 0} \frac{(-1)^n a^{kn} q^{(dk + \frac d2)n^2 -\frac d2 n} 
       (1-aq^{2dn})(a;q^d)_n}
    {(1-a)(q^d;q^d)_n}
\end{equation*}
\end{lem}
\begin{proof}
\begin{eqnarray*}
   && \sum_{n\geqq 0} \frac{(-1)^n a^{kn} q^{(dk + \frac d2)n^2 -\frac d2 n} 
       (1-aq^{2dn})(a;q^d)_n}
    {(1-a)(q^d;q^d)_n} \\
  &=& \sum_{n\geqq 0} \frac{(-1)^n a^{kn} q^{(dk + \frac d2)n^2 -\frac d2 n} 
    (a;q^d)_n} {(1-a)(q^d;q^d)_n} \Big\{ q^{dn}(1-aq^{dn}) + (1-q^{dn}) \Big\} \\
  &=& \sum_{n\geqq 0} \frac{(-1)^n a^{kn} q^{(dk + \frac d2)n^2 +\frac d2 n} 
       (a;q^d)_{n+1}}
    {(1-a)(q^d;q^d)_n} \\
  &&  \qquad+\sum_{n\geqq 1} \frac{(-1)^n a^{kn} q^{(dk + \frac d2)n^2 -\frac d2 n} (aq^d;q^d)_{n-1}}
    {(q^d;q^d)_{n-1}} \\
 &=& \sum_{n\geqq 0} \frac{(-1)^n a^{kn} q^{(dk + \frac d2)n^2 +\frac d2 n} 
       (aq^d;q^d)_n}
    {(q^d;q^d)_n} \\
  &&  \qquad -\sum_{n\geqq 0} \frac{(-1)^n a^{kn+k} 
  q^{(dk + \frac d2)n^2 +(2dk +\frac d2) n + dk} (aq^d;q^d)_n}
    {(q^d;q^d)_n} \\
 &=& \sum_{n\geqq 0} \frac{(-1)^n a^{kn} q^{(dk + \frac d2)n^2 +(\frac d2)n} 
       (aq^d;q^d)_n (1-a^k q^{(2n+1)dk})}
    {(q^d;q^d)_n}\\
 &=& (aq;q)_\infty Q_{d,k,k}(a).
\end{eqnarray*}\openbox 
\end{proof}

\begin{proof}{Proof of \eqref{qdiff1}}
\begin{eqnarray*}
&&\frac{1}{(aq;q)_{d-1}} Q_{d,k,k}(aq^d) \\
&=& \frac{1}{(aq;q)_{d-1} (aq^{d+1};q)_\infty} \\
&& \qquad\times \sum_{n\geqq 0} \frac{(-1)^n a^{kn} q^{(dk + \frac d2)n^2 +(dk-\frac d2) n} 
       (1-aq^{(2n+1)d})(aq^d;q^d)_n}{(1-aq^d)(q^d;q^d)_n} \\ 
&&   \hskip 5cm \mbox{(by Lemma~\ref{Qdkk}) }\\
&=& \frac{1}{(aq;q)_\infty} 
    \sum_{n\geqq 0} \frac{(-1)^n a^{kn} q^{(dk + \frac d2)n^2 +(dk-\frac d2) n} 
       (1-aq^{(2n+1)d})(aq^d;q^d)_n}{(q^d;q^d)_n}\\
  &=&Q_{d,k,1}(a).
\end{eqnarray*}\openbox
\end{proof}

\begin{proof}{Proof of \eqref{qdiff2}}
\begin{eqnarray*}
& & Q_{d,k,i}(a) - Q_{d,k,i-1}(a)\\
&=& \frac{1}{(aq;q)_\infty} 
    \sum_{n=0}^\infty \frac{(-1)^n a^{kn} q^{(k+\frac 12)dn^2 + (k+\frac 12)dn} 
    (aq^d;q^d)_n }{(q^d;q^d)_n} \\
& & \qquad \times\Big( q^{-idn} (1-a^i q^{(2n+1)di}) 
  - q^{dn(1-i)}(1-a^{i-1} q^{(2n+1)d(i-1)}) \Big) \\
&=& \frac{1}{(aq;q)_\infty} 
    \sum_{n=0}^\infty \frac{(-1)^n a^{kn} q^{(k+\frac 12)dn^2 + (k+\frac 12)dn} 
    (aq^d;q^d)_n }{(q^d;q^d)_n} \\
& & \qquad \times\Big( q^{-idn} (1-q^{dn}) + a^{i-1}q^{d(n+1)(i-1)}(1-aq^{d(n+1)}) 
\Big)\\
&=& \frac{1}{(aq;q)_\infty} \Bigg( 
    \sum_{n=1}^\infty \frac{(-1)^n a^{kn} q^{(k+\frac 12)dn^2 + (k-i+\frac 12)dn} 
    (aq^d;q^d)_n } {(q^d;q^d)_{n-1}} \\
& & \qquad +\sum_{n=0}^\infty \frac{(-1)^n a^{kn+i-1}
   q^{(k + \frac 12  )dn^2 + (k + \frac 12)dn +d(n+1)(i-1)} 
        (aq^d;q^d)_{n+1} }
        {(q^d;q^d)_n} \Bigg)\\
&=& \frac{1}{(aq;q)_\infty} \Bigg( 
    -\sum_{n=0}^\infty \frac{(-1)^n a^{kn+k} 
    q^{(k + \frac 12)dn^2 + (3k -i + \frac 32)dn + d(2k-i+1)} (aq^d;q^d)_{n+1} }
      {(q^d;q^d)_{n}} \\
& & \qquad +\sum_{n=0}^\infty \frac{(-1)^n a^{kn+i-1} 
q^{ (k + \frac 12) dn^2 + (k+i-\frac 12)dn + d(i-1)} 
        (aq^d;q^d)_{n+1} }
        {(q^d;q^d)_n} \Bigg)\\ 
&=& \frac{a^{i-1} q^{d(i-1)} (1-aq^d) }{(aq;q)_d (aq^{d+1};q)_\infty} 
    \Bigg( 
    \sum_{n=0}^\infty \frac{(-1)^n a^{kn} 
    q^{(k + \frac 12)dn^2 + (k + i - \frac 12)dn } (aq^{2d};q^d)_{n} }
      {(q^d;q^d)_{n}} \\
& & \qquad -\sum_{n=0}^\infty \frac{(-1)^n a^{kn+k-i+1} 
q^{ (k + \frac 12) dn^2 + (3k-i+\frac 32)dn +2d(k-i+1)} 
        (aq^{2d};q^d)_n }
        {(q^d;q^d)_n} \Bigg)\\ 
&=& \frac{a^{i-1} q^{d(i-1)}}{(aq;q)_{d-1} (aq^{d+1};q)_\infty}   
\sum_{n=0}^\infty \frac{(-1)^n a^{kn} q^{(k + \frac 12)dn^2 + (k+i-\frac 12)dn} (aq^{2d};q^d)_n}
     {(q^d;q^d)_n} \\
& & \qquad\qquad\qquad  \times ( 1-a^{k-i+1} q^{2d(n+1)(k-i+1)})  \\
&=& \frac{a^{i-1} q^{d(i-1)}}{(aq;q)_{d-1}} Q_{d,k,k-i+1} (aq^d).
\end{eqnarray*}\openbox
\end{proof}

\begin{remark} Note that for $1\leqq i\leqq k$, $Q_{d,k,i}(0,q) = 1$ which, together with
\eqref{qdiff1} and \eqref{qdiff2} uniquely determine $Q_{d,k,i}(a,q)$ as a power series
in $a$ and $q$.  In \S\ref{LHS}, we will show that certain functions $F_{d,k,i}(a,q)$
satisfy the same recurrence and initial conditions as the $Q_{d,k,i}(a,q)$ for various
values of $d$ and $k$, thus yielding collections of $a$-RRT identities.  Then, in 
\S\ref{PtnThms}, we will see that generating functions for certain classes of partitions
satisfy those same recurrences and initial conditions, thus providing partition identities.
\end{remark}

Rogers-Ramanujan type identities (in $q$ only) are perhaps more 
aesthetically pleasing than their
$a$-RRT counterparts because their right hand sides are 
expressible as infinite products.  Accordingly, we prove the following
proposition for later use.
\begin{prop}\label{infprod}
\begin{equation}
  Q_{d,k,i}(1) = \frac{(q^{id},q^{(2k-i+1)d},q^{(2k+1)d};q^{(2k+1)d})_\infty}{(q;q)_\infty}
\end{equation}
\end{prop}
\begin{proof}
\begin{eqnarray*}
 &&(q;q)_\infty Q_{d,k,i}(1) \\
 &=& \sum_{n\geqq 0} (-1)^n q^{(dk + \frac d2)n^2 +(k-i+\frac 12)dn} 
       (1- q^{(2n+1)di})\\
 &=&  \sum_{n\geqq 0} (-1)^n q^{(dk + \frac d2)n^2 +(dk-di+\frac d2)n} 
   - (-1)^n q^{(dk + \frac d2)n^2 +(dk+ di+\frac d2 )n +di}\\
 &=&  \sum_{n\geqq 0}  (-1)^n q^{(dk + \frac d2)n^2 -(di-dk-\frac d2)n} 
   + \sum_{n=1}^\infty (-1)^n q^{(dk + \frac d2)n^2 +(di-dk-\frac d2)n}\\ 
 &=& \sum_{n=-\infty}^\infty (-1)^n  q^{(dk + \frac d2)n^2 - (dk-di+\frac d2)n}\\
 &=& (q^{di},q^{2dk-di+ d}, q^{2dk+d}; q^{2dk+d})_\infty \\
 &&  \hskip 1cm\mbox{(by Jacobi's triple product identity~\cite[p. 21, Theorem 2.8]{GEA:top})}
\end{eqnarray*}\openbox
\end{proof}

\subsection{Expressions for the left hand sides and their $q$-difference equations}
\label{LHS}
We now work out the $q$-difference equations associated with the left
hand
sides of various $a$-RRT identities.
\subsubsection{The case $(d,k)=(2,2)$}
\hfil\break
\begin{definition}
 \begin{gather*}
  F_{2,2,1}(a):= F_{2,2,1}(a,q):= 
    \sum_{n=0}^\infty \frac{ a^n q^{\frac 32 n^2 + \frac 32 n}} {(aq;q^2)_{n+1} (q;q)_n }.
   \label{F10_1}\\
  F_{2,2,2}(a):= F_{2,2,2}(a,q):= 
    \sum_{n=0}^\infty \frac{a^n q^{\frac 32 n^2 - \frac 12 n}} {(aq;q^2)_n (q;q)_n} 
    \label{F10_2}\\
  F^*_{2,2,2}(a):= F^*_{2,2,2}(a,q):= 
     \sum_{n=0}^\infty \frac{a^n q^{\frac 32 n^2 + \frac 12 n}} {(aq;q^2)_{n+1} (q;q)_n} 
    \label{F10-_2*}
 \end{gather*}
\end{definition} 
 \begin{lem} $ F_{2,2,2}(a) = F^*_{2,2,2}(a).$
 \end{lem}
 \begin{proof}
   \begin{eqnarray*}
   &&  F_{2,2,2}(a) - F^*_{2,2,2}(a) \\
   &=& \sum_{n=0}^\infty \frac{ a^n q^{\frac 32 n^2 - \frac 12 n}} {(aq;q^2)_{n+1} (q;q)_n}
      \Big( (1-aq^{2n+1}) - q^n \Big) \\
   &=& \sum_{n=0}^\infty \frac{ a^n q^{\frac 32 n^2 - \frac 12 n}} {(aq;q^2)_{n+1} (q;q)_n}
      \Big( (1-aq^{2n+1})(1-q^n) - aq^{3n+1} \Big) \\
    &=& \sum_{n=1}^\infty \frac{ a^n q^{\frac 32 n^2 - \frac 12 n}} {(aq;q^2)_{n} (q;q)_{n-1}}
      - \sum_{n=0}^\infty \frac{ a^{n+1} q^{\frac 32 n^2 + \frac 52 n + 1}} {(aq;q^2)_{n+1} (q;q)_n}
     \\   
    &=& 0  
    \end{eqnarray*}\openbox
 \end{proof}
       
 \begin{lem}\label{Lem2}
The $F_{2,2,i}(a,q)$ satisfy the following $q$-difference equations:
\begin{gather}
 F_{2,2,1}(a) = \frac{1}{1-aq} F_{2,2,2}(aq^2) \label{F10_1qde} \\
 F_{2,2,2}(a) = F_{2,2,1}(a) + \frac{aq^2}{1-aq} F_{2,2,1} (aq^2) \label{F10_2qde} 
\end{gather}
which, together with $F_{2,2,i}(0) = 1$ for $i=1,2$, uniquely determine $F_{2,2,i}(a)$ as
a double power series in $a$ and $q$.
\end{lem}   
 \begin{proof} By inspection, we see
   \[ \frac{1}{1-aq} F^*_{2,2,2} (aq^2) = \frac{1}{1-aq} F_{2,2,2}(aq^2) = 
   F_{2,2,1}(a), \]
   and so (\ref{F10_1qde}) is established.
   Next,
   \begin{eqnarray*}
    F^*_{2,2,2}(a) - F_{2,2,1}(a) &=& \sum_{n=1}^\infty 
    \frac{a^n q^{\frac 32 n^2 + \frac 12 n} (1-q^n)}
    {(aq;q^2)_{n+1} (q;q)_n} \\
    &=&\sum_{n=0}^\infty \frac{a^{n+1} q^{\frac 32 n^2 + \frac 72 n + 2} }
    {(aq;q^2)_{n+2} (q;q)_n} \\
    &=&\frac{aq^2}{1-aq} F_{2,2,1}(aq^2),
   \end{eqnarray*}
  which verifies~(\ref{F10_2qde}).\openbox
 \end{proof}
 
 Thus, by combining Lemma~\ref{Lem2} with Theorem~\ref{qdiffs}, 
 we have established the following theorem:
 \begin{thm}\label{a-mod10}
 For $i=1,2$, 
 \[ F_{2,2,i}(a) = Q_{2,2,i}(a). \]
 \end{thm}
 
Setting $a=1$ and employing Proposition~\ref{infprod}, we obtain 
two identities of Rogers~\cite{Rog1894}, 
which appear as (44) and (46) on Slater's list~\cite{Slater2}:
 \begin{cor}
   \begin{gather}
     \sum_{n=0}^\infty \frac{ q^{\frac 32 n^2 + \frac 32 n}} {(q;q^2)_{n+1} (q;q)_n }= 
     \frac{ (q^2,q^8,q^{10};q^{10})_\infty}{(q;q)_\infty} \label{RogMod10-2}\\
     \sum_{n=0}^\infty \frac{ q^{\frac 32 n^2 - \frac 12 n}} {(q;q^2)_n (q;q)_n} = 
     \sum_{n=0}^\infty \frac{ q^{\frac 32 n^2 + \frac 12 n}} {(q;q^2)_{n+1} (q;q)_n} =
     \frac{ (q^4,q^6,q^{10};q^{10})_\infty}{(q;q)_\infty}. \label{RogMod10-1}
    \end{gather}
  \end{cor}

\subsubsection{The case $(d,k)=(2,3)$}
\hfil\break
\begin{definition}
\begin{gather*} 
  F_{2,3,1}(a) := F_{2,3,1}(a,q) := \sum_{n=0}^\infty \frac{a^n q^{n^2 + 2n}}
  {(aq;q^2)_{n+1} (q;q)_n}\\
  F_{2,3,2}(a) := F_{2,3,2}(a,q) := \sum_{n=0}^\infty \frac{a^n q^{n^2 + n}}
  {(aq;q^2)_{n+1} (q;q)_n}\\
  F_{2,3,3}(a) := F_{2,3,3}(a,q) := \sum_{n=0}^\infty \frac{a^n q^{n^2}}
  {(aq;q^2)_{n} (q;q)_n}\\
\end{gather*}
\end{definition}

\begin{lem}\label{Lem1}
The $F_{2,3,i}(a,q)$ satisfy the following $q$-difference equations:
\begin{gather}
 F_{2,3,1}(a) = \frac{1}{1-aq} F_{2,3,3}(aq^2) \label{F1} \\
 F_{2,3,2}(a) = F_{2,3,1}(a) + \frac{aq^2}{1-aq} F_{2,3,2} (aq^2) \label{F2} \\
 F_{2,3,3}(a) = F_{2,3,2}(a) + \frac{a^2 q^4}{1-aq} F_{2,3,1}(aq^2) \label{F3},
\end{gather}
which, together with $F_{2,3,i}(0) = 1$ for $i=1,2,3$, uniquely determine $F_{2,3,i}(a)$ as
a double power series in $a$ and $q$.
\end{lem}

\begin{proof}
 \[ F_{2,3,1}(a) = \frac{1}{1-aq} F_{2,3,3}(aq^2)\] is clear, so (\ref{F1}) is immediate.
 Next,
 \begin{eqnarray*}
   F_{2,3,2}(a) - F_{2,3,1}(a) &=& \sum_{n=0}^\infty \frac{a^n q^{n^2+n}}{(aq;q^2)_{n+1} (q;q)_n} (1-q^n)\\
                  &=& \sum_{n=1}^\infty \frac{a^n q^{n^2+n}}{(aq;q^2)_{n+1} (q;q)_{n-1}} \\
                  &=& \sum_{n=0}^\infty \frac{a^{n+1} q^{(n+1)^2+(n+1)}}{(aq;q^2)_{n+2} (q;q)_n} \\
                  &=& \frac{aq^2}{1-aq}\sum_{n=0}^\infty 
                   \frac{a^n q^{n^2+3n}}{(aq^3;q^2)_{n+1} (q;q)_n} \\
                  &=& \frac{aq^2}{1-aq} F_{2,3,2}(aq^2),
 \end{eqnarray*}
so (\ref{F2}) is established.
Establishing (\ref{F3}) is a bit trickier, and requires us to define a ``catalyst'' function
 \[ \phi(a) := \sum_{n=0}^\infty \frac{ a^{n+1} q^{n^2 + 3n + 1}}{(aq;q^2)_{n+1} (q;q)_n}.\]
 \begin{eqnarray*}
  & &\frac{a^2 q^4}{1-aq} F_{2,3,1}(aq^2) + \phi(a)\\
  &=&\frac{a^2 q^4}{1-aq} \sum_{n=0}^\infty \frac{a^n q^{n^2+4n}}{ (aq^3;q^2)_{n+1} (q;q)_n}
                     + \sum_{n=0}^\infty \frac{ a^{n+1} q^{n^2 + 3n + 1}}{(aq;q^2)_{n+1} (q;q)_n}\\
  &=& \sum_{n=0}^\infty \frac{a^{n+2} q^{(n+2)^2}}{ (aq;q^2)_{n+2} (q;q)_n}
                     + \sum_{n=0}^\infty \frac{ a^{n+1} q^{n^2 + 3n + 1}}{(aq;q^2)_{n+1} (q;q)_n}\\
  &=&\sum_{n=1}^\infty \frac{a^{n+1} q^{(n+1)^2}}{ (aq;q^2)_{n+1} (q;q)_{n-1}}
                     + \sum_{n=0}^\infty \frac{ a^{n+1} q^{n^2 + 3n + 1}}{(aq;q^2)_{n+1} (q;q)_n}\\
  &=&\sum_{n=0}^\infty \frac{a^{n+1} q^{n^2 + 2n + 1}}{ (aq;q^2)_{n+1} (q;q)_n}(1-q^n)
                     + \sum_{n=0}^\infty \frac{ a^{n+1} q^{n^2 + 3n + 1}}{(aq;q^2)_{n+1} (q;q)_n}\\
  &=&\sum_{n=0}^\infty \frac{a^{n+1} q^{n^2 + 2n + 1}}{ (aq;q^2)_{n+1} (q;q)_n}
    -\sum_{n=0}^\infty \frac{a^{n+1} q^{n^2 + 3n + 1}}{ (aq;q^2)_{n+1} (q;q)_n}
                     + \sum_{n=0}^\infty \frac{ a^{n+1} q^{n^2 + 3n + 1}}{(aq;q^2)_{n+1} (q;q)_n}\\
  &=&\sum_{n=0}^\infty \frac{a^{n+1} q^{n^2 + 2n + 1}}{ (aq;q^2)_{n+1} (q;q)_n}\\
  &=&\sum_{n=0}^\infty \frac{a^n q^{n^2}}{ (aq;q^2)_n (q;q)_{n-1}}\\
  &=&\sum_{n=0}^\infty \frac{a^n q^{n^2}(1-aq^{2n+1})(1-q^n)}{ (aq;q^2)_{n+1} (q;q)_n}\\
  &=&\sum_{n=0}^\infty \frac{a^n q^{n^2}(1-aq^{2n+1}-q^n+aq^{3n+1})}{ (aq;q^2)_{n+1} (q;q)_n}\\
  &=& \sum_{n=0}^\infty \frac{a^n q^{n^2}}{ (aq;q^2)_n (q;q)_n} 
     \left( 1-\frac{q^n}{1-aq^{2n+1}} \right)
   +  \sum_{n=0}^\infty \frac{a^{n+1} q^{n^2+3n+1}}{ (aq;q^2)_{n+1} (q;q)_n}\\
  &=& F_{2,3,3}(a) - F_{2,3,2}(a) + \phi(a),
 \end{eqnarray*}
 and thus (\ref{F3}) is established.\openbox
\end{proof}

Thus, combining Lemma~\ref{Lem1} with Theorem~\ref{qdiffs},
we have established the following theorem:
 \begin{thm}\label{a-mod14}
 For $i=1,2,3$, 
 \[ F_{2,3,i}(a) = Q_{2,3,i}(a). \]
 \end{thm}
 
By setting $a=1$ and employing Proposition~\ref{infprod}, 
we obtain three identities of Rogers (~\cite{Rog1894}
and~\cite{Rog1917}), 
which appear as (59), (60), and (61) respectively on Slater's list~\cite{Slater2}:
\begin{cor}
\begin{gather}
\sum_{n=0}^\infty \frac{q^{n^2 + 2n}}{(q;q^2)_{n+1} (q;q)_n} 
= \frac{(q^2,q^{12},q^{14};q^{14})}{(q;q)_\infty} \label{RogMod14-3}\\
\sum_{n=0}^\infty \frac{q^{n^2 + n}}{(q;q^2)_{n+1} (q;q)_n} 
= \frac{(q^4,q^{10},q^{14};q^{14})}{(q;q)_\infty} \label{RogMod14-2}\\
\sum_{n=0}^\infty \frac{q^{n^2}}{(q;q^2)_{n} (q;q)_n} 
= \frac{(q^6,q^{8},q^{14};q^{14})}{(q;q)_\infty} \label{RogMod14-1}
\end{gather}
\end{cor}

\subsubsection{The case $(d,k)=(2,4)$}
\hfil\break
\begin{definition}
\begin{gather*}
F_{2,4,1}(a):=F_{2,4,1}(a,q):= 
   \sum_{n\geqq 0} \sum_{r\geqq 0} \frac{a^{n+r} q^{n^2 + 2n + 2r^2 + 2r} }
   {(aq;q^2)_{n+1} (q;q)_{n-2r} (q^2;q^2)_r}\\
F_{2,4,2}(a):=F_{2,4,2}(a,q):= 
   \sum_{n\geqq 0} \sum_{r\geqq 0} \frac{a^{n+r} q^{n^2 + 2n + 2r^2 + 2r} (1+aq^{2r+2})}
   {(aq;q^2)_{n+1} (q;q)_{n-2r} (q^2;q^2)_r}\\
F_{2,4,3}(a):=F_{2,4,3}(a,q):= 
   \sum_{n\geqq 0} \sum_{r\geqq 0} \frac{a^{n+r} q^{n^2 + 2r^2 + 2r} }
   {(aq;q^2)_{n} (q;q)_{n-2r} (q^2;q^2)_r}\\   
F_{2,4,4}(a):=F_{2,4,4}(a,q):= 
   \sum_{n\geqq 0} \sum_{r\geqq 0} \frac{a^{n+r} q^{n^2 + 2r^2} }
   {(aq;q^2)_{n} (q;q)_{n-2r} (q^2;q^2)_r}
\end{gather*}
\end{definition}

\begin{lem}\label{LemMod18}
The $F_{2,4,i}(a,q)$ satisfy the following $q$-difference equations:
\begin{gather*}
 F_{2,4,1}(a) = \frac{1}{1-aq} F_{2,4,4}(aq^2)  \\
 F_{2,4,2}(a) = F_{2,4,1}(a) + \frac{aq^2}{1-aq} F_{2,4,3} (aq^2)  \\
 F_{2,4,3}(a) = F_{2,4,2}(a) + \frac{a^2 q^4}{1-aq} F_{2,4,2}(aq^2) \\
 F_{2,4,4}(a) = F_{2,4,3}(a) + \frac{a^3 q^{6}}{1-aq} F_{2,4,1}(aq^2),  
\end{gather*}
which, together with $F_{2,4,i}(0) = 1$ for $i=1,2,3,4$, uniquely determine $F_{2,4,i}(a)$ as
a double power series in $a$ and $q$.
\end{lem}

If the reader has been following along carefully, the details of the calculations should by now 
be routine, so I choose to omit the proof of this and subsequent lemmas establishing
the $q$-difference equations satisfied by the various $F_{d,k,i}(a)$.

Thus, combining Lemma~\ref{LemMod18} with Theorem~\ref{qdiffs},
we have established the following theorem:
 \begin{thm}
 For $i=1,2,3,4$, 
 \[ F_{2,4,i}(a) = Q_{2,4,i}(a) \]
 \end{thm}

By setting $a=1$ and employing Proposition~\ref{infprod}, 
we obtain four new Rogers-Ramanujan type
identities related to the modulus $18$, listed as \eqref{mod18-1}--\eqref{mod18-4} in
the appendix.

\subsubsection{The case $(d,k) = (3,3)$}
\hfil\break
\begin{definition}\begin{gather*}
F_{3,3,1}(a):=
  \sum_{n\geqq 0}\sum_{r\geqq 0} \frac{(-1)^r a^n q^{n^2 + 3n +3r(r-1)/2} 
  (aq^3;q^3)_{n-r}}
   {(aq;q)_{2n+2} (q;q)_{n-3r} (q^3;q^3)_r}\\
F_{3,3,2}(a):=
  \sum_{n\geqq 0}\sum_{r\geqq 0} 
   \frac{(-1)^r a^{n-1} q^{n^2  +3r(r-3)/2} (a;q^3)_{n-r} (1 + aq^{3r}- q^{3r})}
   {(a;q)_{2n} (q;q)_{n-3r} (q^3;q^3)_r}\\
F_{3,3,3}(a):=
  \sum_{n\geqq 0}\sum_{r\geqq 0} \frac{(-1)^r a^n q^{n^2 + 3r(r-1)/2}(a;q^3)_{n-r}}
   {(a;q)_{2n-1} (q;q)_{n-3r} (q^3;q^3)_r}
\end{gather*}
\end{definition}

\begin{lem}\label{LemMod21}
The $F_{3,3,i}(a,q)$ satisfy the following $q$-difference equations:
\begin{gather*}
 F_{3,3,1}(a) = \frac{1}{(1-aq)(1-aq^2)} F_{3,3,3}(aq^3)  \\
 F_{3,3,2}(a) = F_{3,3,1}(a) + \frac{aq^3}{(1-aq)(1-aq^2)} F_{3,3,2} (aq^3)  \\
 F_{3,3,3}(a) = F_{3,3,2}(a) + \frac{a^2 q^6}{(1-aq)(1-aq^2)} F_{3,3,1}(aq^3) ,  
\end{gather*}
which, together with $F_{3,3,i}(0) = 1$ for $i=1,2,3$, uniquely determine $F_{3,3,i}(a)$ as
a double power series in $a$ and $q$.
\end{lem}

Thus, combining Lemma~\ref{LemMod21} with Theorem~\ref{qdiffs},
we have established the following theorem:
 \begin{thm}
 For $i=1,2,3$, 
 \[ F_{3,3,i}(a) = Q_{3,3,i}(a). \]
 \end{thm}

As an immediate corollary, by letting $a\to 1$, we obtain three new Rogers-Ramanujan type
identities related to the modulus $21$, listed as \eqref{mod21-1}--\eqref{mod21-3} in
the appendix.

\subsubsection{The case $(d,k) = (3,4)$}
\hfil\break
\begin{definition}
\begin{gather*}
F_{3,4,1}(a):=F_{3,4,1}(a,q):= 
   \sum_{n=0}^\infty \frac{a^n q^{n(n+3)} (aq^3;q^3)_n}{(aq;q)_{2n+2} (q;q)_n}\\
F_{3,4,2}(a):=F_{3,4,2}(a,q):= 
   \sum_{n=0}^\infty \frac{a^n q^{n(n+2)} (aq^3;q^3)_n}{(aq;q)_{2n+2} (q;q)_n}\\
F_{3,4,3}(a):=F_{3,4,3}(a,q):= 
   \sum_{n=0}^\infty \frac{a^n q^{n(n+1)} (aq^3;q^3)_n}{(aq;q)_{2n+1} (q;q)_n}\\
F_{3,4,4}(a):=F_{3,4,4}(a,q):= 
   \sum_{n=0}^\infty \frac{a^n q^{n^2} (a;q^3)_n}{(a;q)_{2n} (q;q)_n}
\end{gather*}
\end{definition}

\begin{lem}
The $F_{3,4,i}(a,q)$ satisfy the following $q$-difference equations:
\begin{gather*}
 F_{3,4,1}(a) = \frac{1}{(1-aq)(1-aq^2)} F_{3,4,4}(aq^3)  \\
 F_{3,4,2}(a) = F_{3,4,1}(a) + \frac{aq^3}{(1-aq)(1-aq^2)} F_{3,4,3} (aq^3)  \\
 F_{3,4,3}(a) = F_{3,4,2}(a) + \frac{a^2 q^6}{(1-aq)(1-aq^2)} F_{3,4,2}(aq^3) \\
 F_{3,4,4}(a) = F_{3,4,3}(a) + \frac{a^3 q^{9}}{(1-aq)(1-aq^2)} F_{3,4,1}(aq^3),  
\end{gather*}
which, together with $F_{3,4,i}(0) = 1$ for $i=1,2,3,4$, uniquely determine $F_{3,4,i}(a)$ as
a double power series in $a$ and $q$.
\end{lem}

\begin{thm}\label{a-mod27}
 For $i=1,2,3,4$, 
 \[ F_{3,4,i}(a) = Q_{3,4,i}(a) \]
\end{thm}

Upon letting $a\to 1$ and employing Proposition~\ref{infprod}, 
we obtain the Bailey-Dyson mod 27 
identities~\cite[p. 434, equations (B1)--(B4)]{Bailey1}, which appear as
(90)--(93) on Slater's list~\cite{Slater2}.
\begin{cor}
\begin{gather}
   \sum_{n=0}^\infty \frac{q^{n(n+3)} (q^3;q^3)_n}{(q;q)_{2n+2} (q;q)_n} 
     = \frac{(q^{3},q^{24},q^{27};q^{27})_\infty}{(q;q)_\infty}\\
   \sum_{n=0}^\infty \frac{q^{n(n+2)} (q^3;q^3)_n}{(q;q)_{2n+2} (q;q)_n}
     = \frac{(q^{6},q^{18},q^{27};q^{27})_\infty}{(q;q)_\infty}\\
   \sum_{n=0}^\infty \frac{q^{n(n+1)} (q^3;q^3)_n}{(q;q)_{2n+1} (q;q)_n}
     = \frac{(q^{9};q^{9})_\infty}{(q;q)_\infty}\\
   1+\sum_{n=1}^\infty \frac{q^{n^2} (q^3;q^3)_{n-1}}{(q;q)_{2n-1} (q;q)_n}
     = \frac{(q^{12},q^{15},q^{27};q^{27})_\infty}{(q;q)_\infty} \label{BaileyMod27-1}
\end{gather}
\end{cor} 

\subsubsection{The case $(d,k)=(3,5)$}
\hfil\break
\begin{definition}
\begin{gather*}
F_{3,5,1}(a):=
   \sum_{n\geqq 0}\sum_{r\geqq 0} \frac{a^{n+r} q^{n^2 + 3r^2 + 3n + 3r} (aq^3;q^3)_{n-r}}
        {(aq;q)_{2n+2} (q;q)_{n-3r} (q^3;q^3)_r} \\
F_{3,5,2}(a):=
   \sum_{n\geqq 0}\sum_{r\geqq 0} 
        \frac{a^{n+r} q^{n^2 + 3r^2 + 3n + 3r} (aq^3;q^3)_{n-r} (1+aq^{3r+3})}
        {(aq;q)_{2n+2} (q;q)_{n-3r} (q^3;q^3)_r} \\ 
F_{3,5,3}(a):=
   \sum_{n\geqq 0}\sum_{r\geqq 0} 
        \frac{a^{n+r-1} q^{n^2 + 3r^2 -3} (a;q^3)_{n-r} (q^{3r} + aq^{6r+3} -1)}
        {(a;q)_{2n} (q;q)_{n-3r} (q^3;q^3)_r} \\ 
F_{3,5,4}(a):=
   \sum_{n\geqq 0}\sum_{r\geqq 0} 
        \frac{a^{n+r} q^{n^2 + 3r^2 +3r} (a;q^3)_{n-r} }
        {(a;q)_{2n} (q;q)_{n-3r} (q^3;q^3)_r} \\                            
F_{3,5,5}(a):=
   \sum_{n\geqq 0}\sum_{r\geqq 0} \frac{a^{n+r} q^{n^2+3r^2} (a;q^3)_{n-r}}
        {(a;q)_{2n} (q;q)_{n-3r} (q^3;q^3)_r}
\end{gather*}
\end{definition}

\begin{lem}
The $F_{3,5,i}(a,q)$ satisfy the following $q$-difference equations:
\begin{gather*}
 F_{3,5,1}(a) = \frac{1}{(1-aq)(1-aq^2)} F_{3,5,5}(aq^3)  \\
 F_{3,5,2}(a) = F_{3,5,1}(a) + \frac{aq^3}{(1-aq)(1-aq^2)} F_{3,5,4} (aq^3)  \\
 F_{3,5,3}(a) = F_{3,5,2}(a) + \frac{a^2 q^6}{(1-aq)(1-aq^2)} F_{3,5,3}(aq^3) \\
 F_{3,5,4}(a) = F_{3,5,3}(a) + \frac{a^3 q^{9}}{(1-aq)(1-aq^2)} F_{3,5,2}(aq^3)\\
 F_{3,5,5}(a) = F_{3,5,4}(a) + \frac{a^4 q^{12}}{(1-aq)(1-aq^2)} F_{3,5,1}(aq^3)  
\end{gather*}
which, together with $F_{3,5,i}(0) = 1$ for $i=1,2,3,4,5$, uniquely determine $F_{3,5,i}(a)$ as
a double power series in $a$ and $q$.
\end{lem}

\begin{thm}
 For $i=1,2,3,4,5$, 
 \[ F_{3,5,i}(a) = Q_{3,5,i}(a) \]
\end{thm}

Upon letting $a\to 1$ and employing Proposition~\ref{infprod}, 
we obtain five mod 33 identities listed in the appendix as
\eqref{mod33-1} through \eqref{mod33-5}.

\section{Partition Theorems} \label{PtnThms}

  In 1961, Basil Gordon~\cite{Gordon} published an infinite family of partition identities which
generalized the combinatorial version of the Rogers-Ramanujan identities:
\begin{GordonThm}
  Let $B_{1,k,i}(n)$ denote the number of partitions of $n$ wherein 
  $1$ appears as a part at most $i-1$ times, and the total number of 
appearances of any
  two consecutive integers $j$ and $j+1$ is at most $k-1$. 
  Let $A_{1,k,i}(n)$ denote the number of partitions of $n$
into parts not congruent to $0$ or $\pm i \pmod{2k+1}$.  Then $A_{1,k,i}(n) = B_{1,k,i}(n)$ for
all $n$ and $1\leqq i \leqq k$.
\end{GordonThm}

Later, George Andrews~\cite{GEA:oddmoduli} found an analytic counterpart to 
Gordon's partition theorem:
 
\begin{AndrewsAnalytic}
For $1\leqq i\leqq k$ and $k\geqq 2$,
\begin{equation}\label{AGI}
 \sum_{ n_1, n_2, \dots, n_{k-1}\geqq 0} \frac{ q^{\sum_{j=1}^{k-1} N_j^2 + \sum_{j=i}^{k-1} N_j}}
 {(q;q)_{n_1} (q;q)_{n_2} \dots (q;q)_{n_{k-1}}} =
 \underset{n\not\equiv 0, \pm i \pmod{2k+1}}{\prod_{n=1}^\infty} \frac{1}{1-q^n},
\end{equation}
where $N_j = \sum_{h=j}^{k-1} n_h$.
\end{AndrewsAnalytic}

  Motivated by the analytic results earlier in this paper, we consider 
Theorem~\ref{GenGordon}, restated here for convenience.

\begin{thm}
Let $A_{d,k,i}(n)$ denote the number of partitions of $n$ into parts 
$\not\equiv 0, \pm di\pmod{2dk+d}$.  Let $B_{d,k,i}(n)$ denote the
number of partitions of $n$ wherein
  \begin{itemize}
     \item $d$ appears as a part at most $i-1$ times,
     \item the total number of appearances of $dj$ and $dj+d$ (i.e. any two consecutive
       multiples of $d$) together is at most $k-1$, and
     \item nonmultiples of $d$ may appear as parts without restriction.
  \end{itemize}
Then for $1\leqq i\leqq k$, $A_{d,k,i}(n) = B_{d,k,i}(n)$.
\end{thm}

\begin{remark} Clearly, the case $d=1$ is Gordon's partition theorem. 
\end{remark}
\begin{proof}  
   \begin{eqnarray*}
     \sum_{n=0}^\infty B_{d,k,i}(n) q^n &=& 
       \underset{d\nmid j}{\prod_{j=1}^\infty} \frac{1}{1-q^j}  \sum_{n=0}^\infty B_{1,k,i}(dn) q^{nd}\\
     &=& \underset{d\nmid j}{\prod_{j=1}^\infty}\frac{1}{1-q^j} \times 
         \underset{j\not\equiv 0, \pm i \pmod{(2k+1)}}{\prod_{j=1}^\infty} \frac{1}{1-q^{dj}}\\
     &=& \underset{j\not\equiv 0, \pm di \pmod{(2k+1)d}}{\prod_{j=1}^\infty} \frac{1}{1-q^j}    
   \end{eqnarray*}\openbox
\end{proof}

\begin{definition}  Let $b_{d,k,i}(m,n)$ denote the number of partitions of $n$ of the kind
enumerated by $B_{d,k,i}(n)$ with the further restriction that the partition contains
exactly $m$ parts.
\end{definition}
\begin{definition}
 \[  \mathcal{B}_{d,k,i}(a):= 
\mathcal{B}_{d,k,i}(a,q):= \sum_{m,n\geqq 0} b_{d,k,i}(m,n) a^m q^n . \]
\end{definition}

\begin{thm} The $\mathcal{B}_{d,k,i} (a)$ satisfy the following system of
$q$-difference equations:
\begin{eqnarray}
  \mathcal{B}_{d,k,1}(a) &=& \frac{1}{(aq;q)_{d-1}} \mathcal{B}_{d,k,k}(aq^d)  \label{B1}\\
   \mathcal{B}_{d,k,i}(a) &=& 
\mathcal{B}_{k,d,i-1}(a) + \frac{a^{i-1} q^{(i-1)d}}{(aq;q)_{d-1}}
    \mathcal{B}_{d,k,k-i+1}(aq^d) \label{B2},
\end{eqnarray} for $2\leqq i \leqq k$.
\end{thm}

\begin{proof}  
  To obtain partitions of the type enumerated by $b_{d,k,1}(m,n)$ from 
those enumerated by $b_{d,k,k}(m,n)$, one simply needs to increase each 
part in the latter class by $d$ and adjoin as many
$1$'s, $2$'s, \dots, and $(d-1)$'s as desired.  Thus, \eqref{B1} holds.
 
 Now let us segregate the partitions generated by $b_{d,k,i}(m,n)$ 
into two classes:  those where $d$ appears as a part at most $i-2$
times and those where $d$ appears exactly $i-1$ times.  
Those in the former class
 are the entire set of partitions enumerated by $b_{d,k,i-1}(m,n)$.  
Those in the latter class may
 be obtained by starting with the set of partitions enumerated by 
$b_{d,k,k-i+1}$, increasing each
 part by $d$, and affixing exactly $i-1$ copies of the part $d$, and as many
$1$'s, $2$'s, \dots, and $(d-1)$'s as desired.  Thus, \eqref{B2} holds.\openbox
\end{proof}

Since $\mathcal{B}_{d,k,i}(0)$ for $1\leqq i\leqq k$, by uniqueness of 
power series, we immediately obtain
\begin{cor}\label{combininterp}
\begin{equation}
   \mathcal{B}_{d,k,i}(a) = Q_{d,k,i}(a)
\end{equation} for all $d$, all $k$, and $1\leqq i \leqq k$, and
\begin{equation} 
   \mathcal{B}_{d,k,i}(a) = Q_{d,k,i}(a) = F_{d,k,i}(a)
\end{equation} for $(d,k) = (2,2), (2,3), (2,4), (3,3), (3,4), (3,5)$, 
where $1\leqq i\leqq k$.
\end{cor}

As a corollary of Corollary~\ref{combininterp}, by setting
$a=1$, and in light of \eqref{infprod}, we obtain combinatorial
interpretations of a variety of identities in Slater's list, as well 
as some of the new identities presented in the appendix.
For example, the statement
  \[ \mathcal{B}_{2,3,i}(1) = Q_{2,3,i}(1) \] provides the partition
theoretic interpretation of the Rogers mod $14$
identities~\eqref{RogMod14-3}--\eqref{RogMod14-1}, which was
stated in the introduction as Corollary~\ref{combmod14}.
Of course, similar partition theoretic statements can be made for all 
other values of $d$ and $k$, and can be seen as the combinatorial 
counterparts to the $a=1$ case of the 
various $F_{d,k,i}(a) = Q_{d,k,i}(a)$ identities presented in
\S\ref{qDiffEqns}.

\section{Conclusion}\label{concl}
This paper was motivated by taking a careful second look at the methods
employed by Bailey (\cite{Bailey1},\cite{Bailey2}) and seeing if they
could be pushed a bit farther.  Notice that only classical techniques 
(Bailey's Lemma, transformations basic hypergeometric series, and
$q$-difference equations) were used. One of the goals of this paper
is to illustrate that even after all these years, many stones remain 
unturned along the Rogers-Ramanujan path, even when only classical
methods are used.

  Presumably the methods of this paper could be used to obtain
additional identities for other values of $d$ and $k$.
For instance if $d+k=7$, the expression for $\beta_m(a,q)$ will 
involve a ${}_{10} W_{9}$, which could be transformed
into a double sum expression (see ~\cite{GEA:ProbPros}), 
ultimately yielding a triple sum--product identity.

  Also, considering the sets of identities produced when instances
of the parametrized Bailey pair in Theorem~\ref{GenBP} 
are inserted into \eqref{ATNSBL}, it seems reasonable 
that the associated identities could be related to a ``$d$-extended''
version of Andrews' combinatorial generalization of the 
G\"ollnitz-Gordon partition theorem~\cite{GEA:GG}, analogous
to Theorem~\ref{GenGordon}.   Likewise, it is plausible
that the identities arising in connection with
\eqref{SSBL} could be explained combinatorially using the 
overpartitions studied recently by Corteel and 
Lovejoy~(\cite{CorteelLovejoy}, \cite{Lovejoy}).

  Furthermore, the technique of obtaining parametrized
Bailey pairs could presumably be applied to other $\alpha$'s such as the
one from which the Rogers-Selberg identities~\cite[p. 5, (ii)]{Bailey2} 
or Bailey's mod 9 identities~\cite[p. 5, (iii)]{Bailey2} are derived, yielding
other families of results.

  Additionally, finite analogs of Rogers-Ramanujan type identities have,
in recent years,
been of great interest in physics~(e.g. \cite{abf:8v},
\cite{bm:cf}, \cite{bmo:poly}, \cite{bms},\cite{sw}, \cite{sow1}, \cite{sow2}, \cite{sow3})  and 
symbolic computation~(e.g. \cite{pp:rr}, \cite{pwz:a=b}, \cite{wz:multiq}, \cite{dz:fa}).  
In a recent paper~\cite{AVS:FiniteRR},
I presented finite analogs for all of the identities in Slater's list.
The conjecture and proof of these polynomial identities relied heavily
on the use of computer algebra~\cite{AVS:RRtools}.  It is therefore
natural to ask whether the techniques successfully employed for 
finitizing the single sum-product identities of Slater's list can
be extended to the double sum identities presented here, and more
generally to arbitrary multisum--product identities.

\section{Acknowledgement}
I thank the referee for a thorough, careful reading of the manuscript, and for
the many helpful comments.
 
\section*{Appendix: A List of Double Sum Identities of the Rogers-Ramanujan Type}
The following are immediate 
consequences of the more general results presented earlier in the paper.

For $(d,k)=(2,1)$, insert (\ref{beta21}) into (\ref{SSBL}): 
\begin{equation} 
\sum_{n\geqq 0}\sum_{r\geqq 0} 
    \frac{(-1)^r q^{3n(n-1)/2 + r^2 - 2nr} }{(q;q^2)_{n} (q^2;q^2)_r (q;q)_{n-2r}}
  = \frac{(q^{2},q^{4},q^{6};q^{6})_\infty}{(q;q)_\infty} = (-q;q)_\infty
\end{equation}

For $(d,k)=(3,3)$, insert (\ref{beta33}) into (\ref{SSBL}):
\begin{gather}
  1 + \sum_{n\geqq 1}\sum_{r\geqq 0} 
    \frac{(-1)^r q^{n(n+1)/2 + 3r(r-1)/2} (-1;q)_n (q^3;q^3)_{n-r-1}}{(q;q)_{2n-1} (q^3;q^3)_r (q;q)_{n-3r}}
  \nonumber\\
  = \frac{(q^{6},q^{6},q^{12};q^{12})_\infty (-q;q)_\infty}{(q;q)_\infty}  
  \label{mod12-SS}
\end{gather}

For $(d,k)=(2,4)$, insert (\ref{beta24}) into (\ref{SSBL}):
\begin{equation} \label{mod14-SS}
\sum_{n\geqq 0}\sum_{r\geqq 0} 
    \frac{q^{n(n+1)/2 + 2r^2} (-1;q)_n }{(q;q^2)_n (q^2;q^2)_r (q;q)_{n-2r}}
  = \frac{(q^{7},q^{7},q^{14};q^{14})_\infty (-q;q)_\infty}{(q;q)_\infty}
\end{equation}

\dots into (\ref{WBL}):
\begin{equation} \label{mod18-1}
\sum_{n\geqq 0}\sum_{r\geqq 0} 
    \frac{q^{n^2 + 2n + 2r^2 + 2r} }{(q;q^2)_{n+1} (q^2;q^2)_r (q;q)_{n-2r}}
  = \frac{(q^{2},q^{16},q^{18};q^{18})_\infty}{(q;q)_\infty}
\end{equation}

\begin{equation} \label{mod18-2}
\sum_{n\geqq 0}\sum_{r\geqq 0} 
    \frac{q^{n^2 + 2n + 2r^2 + 2r} (1+q^{2r+2})}{(q;q^2)_{n+1} (q^2;q^2)_r (q;q)_{n-2r}}
  = \frac{(q^{4},q^{14},q^{18};q^{18})_\infty}{(q;q)_\infty}
\end{equation}

\begin{equation} \label{mod18-3}
\sum_{n\geqq 0}\sum_{r\geqq 0} 
    \frac{q^{n^2 + 2r^2 + 2r} }{(q;q^2)_{n} (q^2;q^2)_r (q;q)_{n-2r}}
  = \frac{(q^{6},q^{12},q^{18};q^{18})_\infty}{(q;q)_\infty}
\end{equation}

\begin{equation} \label{mod18-4}
\sum_{n\geqq 0}\sum_{r\geqq 0} 
    \frac{q^{n^2 + 2r^2} }{(q;q^2)_n (q^2;q^2)_r (q;q)_{n-2r}}
  = \frac{(q^{8},q^{10},q^{18};q^{18})_\infty}{(q;q)_\infty}
\end{equation}

For $(d,k)=(3,3)$, insert (\ref{beta33}) into (\ref{WBL}):
\begin{equation} \label{mod21-1}
 \sum_{n\geqq 0}\sum_{r\geqq 0} 
    \frac{(-1)^r q^{n^2 + 3n + 3r(r-1)/2} (q^3;q^3)_{n-r}}{(q;q)_{2n+2} (q^3;q^3)_r (q;q)_{n-3r}}
  = \frac{(q^{3},q^{18},q^{21};q^{21})_\infty}{(q;q)_\infty}
\end{equation}

\begin{equation} \label{mod21-2}
  1 + \sum_{n\geqq 1}\sum_{r\geqq 0} 
    \frac{(-1)^r q^{n^2 + 3r(r-3)/2} (q^3;q^3)_{n-r-1}}{(q;q)_{2n-1} (q^3;q^3)_r (q;q)_{n-3r}}
  = \frac{(q^{6},q^{15},q^{21};q^{21})_\infty}{(q;q)_\infty}
\end{equation}

\begin{equation} \label{mod21-3}
  1 + \sum_{n\geqq 1}\sum_{r\geqq 0} 
    \frac{(-1)^r q^{n^2 + 3r(r-1)/2} (q^3;q^3)_{n-r-1}}{(q;q)_{2n-1} (q^3;q^3)_r (q;q)_{n-3r}}
  = \frac{(q^{9},q^{12},q^{21};q^{21})_\infty}{(q;q)_\infty}
\end{equation}

\dots into (\ref{ATNSBL}):
\begin{equation} \label{mod24-2}
  1 + \sum_{n\geqq 1}\sum_{r\geqq 0} 
    \frac{(-1)^r q^{n^2 + 3r(r-3)} (-q;q^2)_n (q^6;q^6)_{n-r-1}}{(q^2;q^2)_{2n-1} (q^6;q^6)_r (q^2;q^2)_{n-3r}}
  = \frac{(q^{3},q^{21},q^{24};q^{24})_\infty (-q;q^2)_\infty}{(q^2;q^2)_\infty}
\end{equation}

\begin{equation} \label{mod24-3}
  1 + \sum_{n\geqq 1}\sum_{r\geqq 0} 
    \frac{(-1)^r q^{n^2 + 3r(r-1)} (-q;q^2)_n (q^6;q^6)_{n-r-1}}{(q^2;q^2)_{2n-1} (q^6;q^6)_r (q^2;q^2)_{n-3r}}
  = \frac{(q^{9},q^{15},q^{24};q^{24})_\infty (-q;q^2)_\infty}{(q^2;q^2)_\infty}
\end{equation}

For $(d,k)=(3,5)$, insert (\ref{beta35}) into (\ref{SSBL}):
\begin{equation} \label{mod24-SS}
  1 + \sum_{n\geqq 1}\sum_{r\geqq 0} 
    \frac{q^{n(n+1)/2 + 3r^2} (-1;q)_n (q^3;q^3)_{n-r-1}}{(q;q)_{2n-1} (q^3;q^3)_r (q;q)_{n-3r}}
  = \frac{(q^{12},q^{12},q^{24};q^{24})_\infty (-q;q)_\infty}{(q;q)_\infty}
\end{equation}

For $(d,k)=(2,4)$, insert (\ref{beta24}) into (\ref{ATNSBL}):
\begin{equation} \label{mod28-3}
\sum_{n\geqq 0}\sum_{r\geqq 0} 
    \frac{q^{n^2 + 4r^2 + 4r} }{(q;q^2)_{n} (q^4;q^4)_r (q^2;q^2)_{n-2r}}
  = \frac{(q^{8},q^{20},q^{28};q^{28})_\infty (-q;q^2)_\infty}{(q^2;q^2)_\infty}
\end{equation}

\begin{equation} \label{mod28-4}
\sum_{n\geqq 0}\sum_{r\geqq 0} 
    \frac{q^{n^2 + 4 r^2} }{(q;q^2)_n (q^4;q^4)_r (q^2;q^2)_{n-2r}}
  = \frac{(q^{12},q^{16},q^{28};q^{28})_\infty (-q;q^2)_\infty}{(q^2;q^2)_\infty}
\end{equation}

For $(d,k)=(3,5)$, insert (\ref{beta35}) into (\ref{WBL}):
\begin{equation} \label{mod33-1}
  \sum_{n\geqq 0}\sum_{r\geqq 0} 
    \frac{q^{n^2 + 3n + 3r^2 + 3r} (q^3;q^3)_{n-r}}{(q;q)_{2n+2} (q^3;q^3)_r (q;q)_{n-3r}}
  = \frac{(q^{3},q^{30},q^{33};q^{33})_\infty}{(q;q)_\infty}
\end{equation}

\begin{equation} \label{mod33-2}
  \sum_{n\geqq 0}\sum_{r\geqq 0} 
    \frac{q^{n^2 + 3n + 3r^2 + 3r} (q^3;q^3)_{n-r} (1+q^{3r+3})}
    {(q;q)_{2n+2} (q^3;q^3)_r (q;q)_{n-3r}}
  = \frac{(q^{6},q^{27},q^{33};q^{33})_\infty}{(q;q)_\infty}
\end{equation}

\begin{equation} \label{mod33-3}
  1 + \sum_{n\geqq 1}\sum_{r\geqq 0} 
    \frac{q^{n^2 + 3r^2 - 3} (q^3;q^3)_{n-r-1} (q^{3r} + q^{6r+3} - 1) }
    {(q;q)_{2n-1} (q^3;q^3)_r (q;q)_{n-3r}}
  = \frac{(q^{9},q^{24},q^{33};q^{33})_\infty}{(q;q)_\infty}
\end{equation}

\begin{equation} \label{mod33-4}
  1 + \sum_{n\geqq 1}\sum_{r\geqq 0} 
    \frac{q^{n^2 + 3r^2 + 3r} (q^3;q^3)_{n-r-1}}{(q;q)_{2n-1} (q^3;q^3)_r (q;q)_{n-3r}}
  = \frac{(q^{12},q^{21},q^{33};q^{33})_\infty}{(q;q)_\infty}
\end{equation}

\begin{equation} \label{mod33-5}
  1 + \sum_{n\geqq 1}\sum_{r\geqq 0} 
    \frac{q^{n^2 + 3r^2} (q^3;q^3)_{n-r-1}}{(q;q)_{2n-1} (q^3;q^3)_r (q;q)_{n-3r}}
  = \frac{(q^{15},q^{18},q^{33};q^{33})_\infty}{(q;q)_\infty}
\end{equation}

For $(d,k)=(4,6)$, insert (\ref{beta46even}) and (\ref{beta46odd}) into (\ref{SSBL}):
\begin{gather}
1+\sum_{m\geqq 1}\sum_{r\geqq 0}
   \frac{(-1)^{m+r} q^{m^2 +m + r^2 + r -2mr} (q^4;q^4)_{m+r-1}
 (-1;q)_{2m}}
    {(q;q)_{4m-1} (q;q)_{2r} (q^2;q^2)_{m-r}} \nonumber \\
+\sum_{m\geqq 0}\sum_{r\geqq 0}
   \frac{(-1)^{m+r} q^{m^2 + m + r^2 + r -2mr +1} (q^4;q^4)_{m+r}
(-1;q)_{2m+1}}
    {(q;q)_{4m+1} (q;q)_{2r+1} (q^2;q^2)_{m-r}} \nonumber \\
=\frac{(q^{18},q^{18},q^{36};q^{36})_\infty(-q;q)_\infty}{(q;q)_\infty} \label{mod36}
\end{gather}

For $(d,k)=(3,5)$, insert (\ref{beta35}) into (\ref{ATNSBL}):
\begin{gather} 
  1 + \sum_{n\geqq 1}\sum_{r\geqq 0} 
    \frac{q^{n^2 + 6r^2 - 6} (-q;q^2)_n (q^6;q^6)_{n-r-1} (q^{6r} + q^{12r+6} - 1 )}
    {(q^2;q^2)_{2n-1} (q^6;q^6)_r (q^2;q^2)_{n-3r}} \nonumber \\
  = \frac{(q^{9},q^{39},q^{48};q^{48})_\infty (-q;q^2)_\infty}
{(q^2;q^2)_\infty}
\label{mod48-3}
\end{gather}

\begin{equation} \label{mod48-4}
  1 + \sum_{n\geqq 1}\sum_{r\geqq 0} 
    \frac{q^{n^2 + 6r^2 + 6r} (-q;q^2)_n (q^6;q^6)_{n-r-1}}{(q^2;q^2)_{2n-1} (q^6;q^6)_r 
    (q^2;q^2)_{n-3r}}
  = \frac{(q^{15},q^{33},q^{48};q^{48})_\infty (-q;q^2)_\infty}{(q^2;q^2)_\infty}
\end{equation}

\begin{equation} \label{mod48-5}
  1 + \sum_{n\geqq 1}\sum_{r\geqq 0} 
    \frac{q^{n^2 + 6r^2} (-q;q^2)_n (q^6;q^6)_{n-r-1}}{(q^2;q^2)_{2n-1} (q^6;q^6)_r (q^2;q^2)_{n-3r}}
  = \frac{(q^{21},q^{27},q^{48};q^{48})_\infty (-q;q^2)_\infty}{(q^2;q^2)_\infty}
\end{equation}

For $(d,k)=(4,6)$, insert (\ref{beta46even}) and (\ref{beta46odd}) into (\ref{WBL}):
\begin{gather}
1+\sum_{m\geqq 1}\sum_{r\geqq 0}
   \frac{(-1)^{m+r}  q^{3m^2 + r^2 + r -2mr} (q^4;q^4)_{m+r-1}}
    {(q;q)_{4m-1} (q;q)_{2r} (q^2;q^2)_{m-r}} \nonumber \\
+\sum_{m\geqq 0}\sum_{r\geqq 0}
   \frac{(-1)^{m+r}  q^{3m^2 + 2m + r^2 + r -2mr +1} (q^4;q^4)_{m+r}}
    {(q;q)_{4m+1} (q;q)_{2r+1} (q^2;q^2)_{m-r}} \nonumber \\
=\frac{(q^{24},q^{28},q^{52};q^{52})_\infty}{(q;q)_\infty}
   \label{mod52}
\end{gather}

\dots into (\ref{ATNSBL}):
\begin{gather}
1+\sum_{m\geqq 1}\sum_{r\geqq 0}
   \frac{(-1)^{m+r} q^{2m^2 + 2r^2 + 2r -4mr}(q^8;q^8)_{m+r-1} (-q;q^2)_{2m}}
    {(q^2;q^2)_{4m-1} (q^2;q^2)_{2r} (q^4;q^4)_{m-r}} \nonumber \\
+\sum_{m\geqq 0}\sum_{r\geqq 0}
   \frac{(-1)^{m+r} q^{2m^2 + 2r^2 + 2r -4mr +1} (q^8;q^8)_{m+r}
  (-q;q^2)_{2m+1}}
    {(q^2;q^2)_{4m+1} (q^2;q^2)_{2r+1} (q^4;q^4)_{m-r}} \nonumber \\
=\frac{(q^{32},q^{40},q^{72};q^{72})_\infty (-q;q^2)_\infty}
   {(q^2;q^2)_\infty} \label{mod72}
\end{gather}

\end{document}